\documentclass[a4paper,12pt]{article}
\usepackage[utf8x]{inputenc}
\usepackage{graphicx}
\usepackage[T1]{fontenc}
\usepackage{amsmath}
\usepackage{amsfonts}
\usepackage{amssymb}
\usepackage{amsthm}
\usepackage{amstext}
\usepackage{pstricks}

\newdimen\CdotAxis
\newcommand*{\CdotAux}[3]{%
  {%
    \settoheight\CdotAxis{$#2\vcenter{}$}%
    \sbox0{%
      \raisebox\CdotAxis{%
        \scalebox{#1}{%
          \raisebox{-\CdotAxis}{%
            $\mathsurround=0pt #2#3$%
          }%
        }%
      }%
    }%
    \dp0=0pt %
    \sbox2{$#2\bullet$}%
    \ifdim\ht2<\ht0 %
      \ht0=\ht2 %
    \fi
    \sbox2{$\mathsurround=0pt #2#3$}%
    \hbox to \wd2{\hss\usebox{0}\hss}%
  }%
}

\newtheorem{theorem}{Theorem}[section]
\newtheorem{lemma}[theorem]{Lemma}
\newtheorem{corollary}[theorem]{Corollary}
\newtheorem{proposition}[theorem]{Proposition}

\theoremstyle{definition}

\theoremstyle{remark}

\def \cA {\mathcal{A}}
\def \cB {\mathcal{B}}

\def \cD {\mathcal{D}}
\def \cE {\mathcal{E}}

\def \cK {\mathcal{K}}

\def \cN {\mathcal{N}}

\def \cP {\mathcal{P}}

\def \cR {\mathcal{R}}

\def \cT {\mathcal{T}}

\def \cW {\mathcal{W}}

\def \a {\alpha}
\def \b {\beta}
\def \g {\gamma}
\def \d {\delta}
\def \e {\varepsilon}

\def \t {\theta}

\def \k {\kappa}

\def \s {\sigma}

\def \T {\Theta}

\def \dD {\mathbb{D}}

\def \N {\mathbb{N}}

\def \R {\mathbb{R}}

\def \lra {\longrightarrow}

\def \Ra {\Rightarrow}
 
\def \Lra {\Longrightarrow}

\def \Ot {(O_n)_{n\geq 0}}

\def \Ttl {(\T^{\ell}_n)_{n\geq 0}}
\def \Ttk {(\T^1_n)_{n\geq 0}}

\def \Ott {(O^{\t}_n)_{n\geq 0}}
\def \Otl {(O^{\ell}_n)_{n\geq 0}}
\def \Otk {(O^{K+1}_n)_{n\geq 0}}
\def \Ztt {(Z^{\t}_n)_{n\geq 0}}
\def \Ztl {(Z^{\ell}_n)_{n\geq 0}}
\def \Ztk {(Z^{K+1}_n)_{n\geq 0}}

\def \pml {\cP^m_{\ell +1}}
\def \zl {\{\, 0,\dots,\ell \,\}}

\def \zk {\{\, 0,\dots,K \,\}}
\def \zm {\{\, 0,\dots,m \,\}}
\def \ul {\{\, 1,\dots,\ell \,\}}

\def \um {\{\, 1,\dots,m \,\}}
\def \otex {o^\theta_{\text{exit}}}
\def \oten {o^\theta_{\text{enter}}}
\def \olen {o^\ell_{\text{enter}}}
\def \olex {o^\ell_{\text{exit}}}
\def \oken {o^{K+1}_{\text{enter}}}
\def \okex {o^{K+1}_{\text{exit}}}
\def \cW {{\mathcal W}^*}

\def \exa {e^{-a}}

\def\uro{\smash{{U}^{\!\!\!\!\raise5pt\hbox{$\scriptstyle o$}}}}
\newcommand*{\lcdot}{\raisebox{-0.25ex}{\scalebox{1.2}{$\cdot$}}}

\begin{document}

\begin{center}
\begin{LARGE}
The distribution of the quasispecies\\[-2 pt]
for the Wright--Fisher model\\[3 pt]
on the sharp peak landscape
\end{LARGE}

\begin{large}
Joseba Dalmau

\vspace{-12pt}
Universit\'e Paris Sud and ENS Paris

\vspace{4pt}
\today
\end{large}
\end{center}

\begin{abstract}
\noindent
We consider the classical Wright--Fisher model
with mutation and selection.
Mutations occur independently in each locus,
and selection is performed according to the sharp peak landscape.
In the asymptotic regime studied in~\cite{CerfWF},
a quasispecies is formed.
We find explicitly the distribution of this quasispecies,
which turns out to be the same distribution as for the Moran model.
\end{abstract}

\section{Introduction}
The concept of quasispecies first appeared in 1971,
in Manfred Eigen's celebrated paper~\cite{Eigen1}.
Eigen studied the evolution of a population of macromolecules,
subject to both selection and mutation effects.
The selection mechanism
is coded in a fitness landscape;
while many interesting landscapes might be considered,
some have been given more attention than others.
One of the most studied landscapes is the sharp peak landscape:
one particular sequence---the master sequence---replicates faster than the rest,
all the other sequences having the same replication rate.
A major discovery made by Eigen 
is the existence of an error threshold for the mutation rate
on the sharp peak landscape:
there is a critical mutation rate $q_c$ such that,
if $q>q_c$ then the population evolves towards a disordered state,
while if $q<q_c$ then the population evolves so as to form a quasispecies,
i.e., a population consisting of a positive concentration of the master sequence,
along with a cloud of mutants which highly resemble the master sequence.

Eigen's model is a deterministic model,
the population of macromolecules is considered to be infinite
and the evolution of the concentrations of the different genotypes
is driven by a system of differential equations.
Therefore, 
when trying to apply the concepts of error threshold and quasispecies
to other areas of biology (e.g. population genetics or virology),
Eigen's model is not particularly well suited;
a model for a finite population,
which incorporates stochastic effects,
is the most natural mathematical approach to the matter.

Several works have tackled the issue of creating
a finite and stochastic version of Eigen's model
\cite{AF}, \cite{DSS}, \cite{DSV}, \cite{Gillespie}, 
\cite{McCaskill}, \cite{Musso}, \cite{NS}, \cite{PEM}, \cite{SAA1}.
Some of these works have recovered the error threshold phenomenon 
in the case of finite populations:
Alves and Fontantari~\cite{AF} find a relation between the error threshold and
the population size by considering a finite version of Eigen's model on the 
sharp peak landscape.
Demetrius, Schuster and Sigmund~\cite{DSS} generalise the error threshold criteria
by modelling the evolution of a population via branching processes.
Nowak and Schuster~\cite{NS} also find the error threshold phenomenon in finite 
populations by making use of a birth and death chain.
Some other works have tried to prove the validity of 
Eigen's model in finite populations by designing algorithms that give
similar results to Eigen's theoretical calculations \cite{Gillespie},
while others have focused on proposing finite population models that
converge to Eigen's model in the infinite population limit \cite{DSV}, \cite{Musso}.

The Wright--Fisher model
is one of the most classical models in mathematical evolutionary theory,
it is also used to understand the evolution of DNA sequences
(see \cite{Durrett}).
In~\cite{CerfWF},
some counterparts of the results on Eigen's model were derived
in the context of the Wright--Fisher model.
The Wright--Fisher model describes 
the evolution of a population of $m$ chromosomes
of length $\ell$ over an alphabet with $\k$ letters.
Mutations occur independently at each locus with probability $q$.
The sharp peak landscape is considered:
the master sequence replicates at rate $\s>1$,
while all the other sequences replicate at rate 1.
The following asymptotic regime is studied:
$$\displaylines{
\ell\to +\infty\,,\qquad m\to +\infty\,,\qquad q\to 0\,,\cr
{\ell q} \to a\,,
\qquad\frac{m}{\ell}\to\alpha\,.}$$
In this asymptotic regime
the error threshold phenomenon present in Eigen's model is recovered,
in the form of a critical curve $\a\psi(a)=\ln\k$
in the parameter space $(a,\a)$.
If $\a\psi(a)<\ln\k$, then
the equilibrium population is totally random,
whereas a quasispecies is formed when
$\a\psi(a)>\ln\k$.
In the regime where a quasispecies is formed,
the concentration of the master sequence in the equilibrium population
is also found.
The aim of this paper is to continue with the study 
of the Wright--Fisher model in the above asymptotic regime
in order to find the distribution of the whole quasispecies.
It turns out that the resulting distribution 
is the same as the one found for the Moran model in~\cite{CD}.
Nevertheless,
the techniques we use to prove our result are very different from those of~\cite{CD}.
The study of the Moran model relied strongly on monotonicity arguments,
and the result was proved inductively.
The initial case and the inductive step boiled down to the study
of birth and death Markov chains,
for which explicit formulas could be found.
The 
Wright--Fisher model is a model with no overlapping generations,
for which this approach is no longer suitable.
In order to find a more robust approach,
we rely on the ideas developed by Freidlin and Wentzell
to investigate random perturbations of dynamical systems~\cite{FW},
as well as some techniques already used in~\cite{CerfWF}.
Our setting is essentially the same as the one in~\cite{CerfWF},
the biggest difference being that we work in several dimensions
instead of having one dimensional processes.
The main challenge is therefore to extend the arguments from~\cite{CerfWF}
to the multidimensional case.
This is achieved by replacing the monotonicity arguments employed in~\cite{CerfWF}
by uniform estimates.

We present the main result in the next section.
The rest of the paper is devoted to the proof.

\section{Main Result}
We present the main result of the article here.
We start by describing the Wright--Fisher model,
we state the result next,
and we give a sketch of the proof at the end of the section.

\subsection{The Wright--Fisher model}
Let $\cA$ be a finite alphabet
and let $\k$ be its cardinality.
Let $\ell,m\geq 1$.
Elements of $\cA^\ell$ represent the chromosome of an individual,
and we consider a population of $m$ such chromosomes.
Two main forces drive the evolution of the population:
selection and mutation.
The selection mechanism is controlled by a fitness function $A:\cA^\ell\to[0,+\infty[\,$.
We define a selection function $F:\cA^\ell\times(\cA^\ell)^m\to[0,1]$
by setting
$$
\forall u\in\cA^\ell\quad \forall x\in(\cA^\ell)^m\qquad
F(u,x)\,=\,\frac{A(u)\text{card}\{\,i:1\leq i\leq m,\, x(i)=u\,\}}{A(x(1))+\cdots+A(x(m))}\,.
$$
For a given population $x$,
the value $F(u,x)$
is the probability that the individual $u$ is chosen when sampling from $x$.
Throughout the replication process, 
mutations occur independently on each allele with probability $q\in\,]0,1-1/\k[\,$.
When a mutation occurs, the letter is replaced by a new letter,
chosen uniformly at random among the remaining $\k-1$ letters of the alphabet.
The mutation mechanism is encoded in a mutation matrix
$M(u,v)$, $u,v\in\cA^\ell$.
The analytical formula for the mutation matrix is as follows:
$$\forall u,v\in\cA^\ell\qquad
M(u,v)\,=\,
\prod_{j=1}^\ell \bigg(
(1-q)1_{u(j)=v(j)}+\frac{q}{\k-1}1_{u(j)\neq v(j)}
\bigg)\,.$$
We consider the classical Wright--Fisher model.
The transition mechanism from one generation to the next one is divided in two steps.
Firstly,
we sample with replacement $m$ chromosomes from the current population,
according to the selection function $F$ given above.
Secondly,
each of the sampled chromosomes mutates according to the law
given by the mutation matrix.
Finally,
the whole old generation is replaced with the new one,
so generations do not overlap.
For $n\geq 0$,
we denote by $X_n$
the population at time $n$,
or equivalently,
the $n$--th generation.
The Wright--Fisher model is the Markov chain $(X_n)_{n\geq 0}$
with state space $(\cA^\ell)^m$,
having the following transition matrix:
\begin{multline*}
\forall n\in\N\quad \forall x,y\in(\cA^\ell)^m\\
P(X_{n+1}=y\,|\, X_n=x)\,=\,
\prod_{i=1}^m\bigg(
\sum_{u\in\cA^\ell}F(u,x)M(u,y(i))
\bigg)\,.\hfil
\end{multline*}

\subsection{Main result}
We will work only with the sharp peak landscape:
there exists a sequence $w^*\in\cA^\ell$,
called master sequence,
whose fitness is $A(w^*)=\s>1$,
whereas for all $u\neq w^*$ in $\cA^\ell$
the fitness $A(u)$ is 1.
We introduce Hamming classes in the space $\cA^\ell$.
The Hamming distance between two chromosomes
$u,v\in\cA^\ell$ is defined as follows:
$$d_H(u,v)\,=\,\text{card}\lbrace\,
i\in\ul : u(i)\neq v(i)
\,\rbrace\,.$$
For $k\in\ul$ and a population $x\in(\cA^\ell)^m$, 
we denote by $N_k(x)$
the number of sequences in the population $x$
which are at distance $k$ from the master sequence, i.e.,
$$N_k(x)\,=\,\text{card}\lbrace\,
i\in\um: d_H(x(i),w^*)=k
\,\rbrace\,.$$
Let us denote by $I(p,t)$
the rate function governing the large deviations of a binomial law
of parameter $p\in[0,1]$:
$$\forall t\in [0,1]\qquad I(p,t)\,=\,
t\ln\frac{t}{p}+(1-t)\ln\frac{1-t}{1-p}\,.$$
We define, for $a\in\,]0,+\infty[\,$,
\begin{multline*}
\hfil\forall k\geq 0\qquad 
\rho^*_k\,=\,
(\s\exa-1)\frac{a^k}{k!}\sum_{i\geq 1}\frac{i^k}{\s^i}\,,\\
\rho^*(a)\,=\,\begin{cases}
\quad \rho^*_0 & \quad\text{if }\ \s\exa>1\\
\quad 0 & \quad\text{if }\ \s\exa\leq 1
\end{cases}\\
\psi(a)\,=\,\inf_{l\in\N}\inf\Bigg\lbrace\,
\sum_{k=1}^{l-1} I\bigg(
\frac{\s\rho_k}{(\s-1)\rho_k-1},\g_k
\bigg)
+\g_k I\bigg(
\exa,\frac{\rho_{k+1}}{\g_k}
\bigg):\\
\rho_0=\rho^*(a),\, 
\rho_l=0,\,
\rho_k,\g_k\in[0,1]\ \text{for}\ 0\leq k<l
\,\Bigg\rbrace\,.
\end{multline*}
\begin{theorem}\label{main}
We suppose that
$$\ell\to+\infty\,,\qquad
m\to+\infty\,,\qquad
q\to0\,,$$
in such a way that
$$\ell q\to a\in\,]0,+\infty[\,\,,\qquad
\frac{m}{\ell}\to\a\in[0,+\infty]\,.$$
We have the following dichotomy:

$\bullet$ if $\a\psi(a)<\ln\k$, then
$$\forall k\geq 0\qquad
\lim_{\genfrac{}{}{0pt}{1}{\ell,m\to\infty,\,q\to 0}{{\ell q} \to a,\,\frac{m}{\ell}\to\a}}\,\lim_{n\to\infty}\,
E\bigg(
\frac{N_k(X_n)}{m}
\bigg)\,=\,0\,,$$

$\bullet$ if $\a\psi(a)>\ln\k$, then
$$\forall k\geq 0\qquad
\lim_{\genfrac{}{}{0pt}{1}{\ell,m\to\infty,\,q\to 0}{{\ell q} \to a,\,\frac{m}{\ell}\to\a}}\,\lim_{n\to\infty}\,
E\bigg(
\frac{N_k(X_n)}{m}
\bigg)\,=\,\rho^*_k\,.$$

Moreover, in both cases,
$$\forall k\geq 0\qquad
\lim_{\genfrac{}{}{0pt}{1}{\ell,m\to\infty,\,q\to 0}{{\ell q} \to a,\,\frac{m}{\ell}\to\a}}\,\lim_{n\to\infty}\,
\text{Var}\bigg(
\frac{N_k(X_n)}{m}
\bigg)\,=\,0\,.$$
\end{theorem}

\subsection{Sketch of proof}
The Wright--Fisher process $(X_n)_{n\geq 0}$
is hard to handle,
mainly due to the huge size of the state space
and the lack of a natural ordering in it.
Instead of directly working with the Wright--Fisher process,
we work with the occupancy process $\Ot$.
The occupancy process is
a simpler process which derives directly from the original process $(X_n)_{n\geq 0}$,
but only keeps the information we are interested in,
namely, the number of chromosomes in each of the 
$\ell+1$ Hamming classes.
The state space of the occupancy process is much simpler than
that of the Wright--Fisher process,
and it is endowed with a partial ordering.
The occupancy process will be the main subject of our study.

We fix next $K\geq 0$ and we focus on finding
the concentration of the individuals in the $K$--th Hamming class.
We compare the time that the occupancy process spends having
at least one individual in one of the Hamming classes $0,\dots,K$ (persistence time), with the time the process spends having no sequences in any of the classes $0,\dots,K$ (discovery time).
Asymptotically,
when $\a\psi(a)<\ln\k$,
the persistence time becomes negligible with respect to the discovery time,
whereas when $\a\psi(a)>\ln\k$,
it is the discovery time that becomes negligible with respect to the persistence time.
This fact, 
which already proves the first assertion of theorem~\ref{main},
is shown in~\cite{CerfWF} for the case $K=0$;
the more general case $K\geq 1$
is dealt with in the same way as the case $K=0$,
and the proof does not make any new contributions to the understanding of the model.
Therefore, 
we will admit this fact and focus on the interesting case $\a\psi(a)>\ln\k$.

We build a coupling to compare the occupancy process with some simpler processes,
which will only keep track of the dynamics of the Hamming classes $0,\dots,K$.
The simpler processes can be viewed as random perturbations of the same dynamical system.
The dynamical system has two fixed points:
an unstable one, 0,
and a stable one, $\rho^*=(\rho^*_0,\dots,\rho^*_K)$.
We use the theory developed by Freidlin and Wentzell~\cite{FW},
as well as some useful estimates from~\cite{CerfWF},
to show that the perturbed processes spend the greatest part of their time
very close to the stable fixed point $\rho^*$,
thus showing that the invariant measure of the perturbed processes
converge to the Dirac mass in $\rho^*$.

\subsection{The occupancy process}
The occupancy process
$\Ot$
will be the starting point of our study.
It is obtained from the original Wright--Fisher process
$(X_n)_{n\geq 0}$
by using a technique known as lumping
(section 4 of \cite{CerfWF}).
Let $\pml$ be the set of the ordered partitions 
of the integer $m$ in at most $\ell+1$ parts:
$$
\pml\,=\,
\big\lbrace\,
(o(0),\dots,o(\ell))\in\N^{\ell+1}:
o(0)+\cdots+o(\ell)=m
\,\big\rbrace\,.
$$
A partition $(o(0),\dots,o(\ell))$
is interpreted as an occupancy distribution,
which corresponds to a population with $o(l)$ 
individuals in the Hamming class $l$, for $0\leq l\leq \ell$.
The occupancy process $\Ot$
is a Markov chain with values in $\pml$
and transition matrix given by:
\begin{multline*}
\forall o,o'\in\pml\\
p_O(o,o')\,=\,
\prod_{0\leq h\leq\ell}\Bigg(
\frac{\sum_{k\in\zl}o(k)A_H(k)M_H(k,h)}{\sum_{h\in\zl}o(h)A_H(h)}
\Bigg)^{o'(h)}\,,\hfil
\end{multline*}
where $A_H$ is the lumped fitness function, 
defined as follows
$$\forall b \in \zl\qquad
A_H(b)\,=\,
\begin{cases}
\quad \s\quad & \text{if } b=0\,,\\
\quad 1\quad & \text{if } b\geq 1\,,
\end{cases}
$$
and $M_H$ is the lumped mutation matrix:
for $b,c\in\zl$ the coefficient $M_H(b,c)$ is given by
$$
\sum_{
\genfrac{}{}{0pt}{1}{0\leq k\leq\ell-b}{
\genfrac{}{}{0pt}{1}
 {0\leq l\leq b}{k-l=c-b}
}
}
{ \binom{\ell-b}{k}}
{\binom{b}{l}}
q^k
(1-q)^{\ell-b-k}
\left(\frac{q}{\kappa-1}\right)^l
\bigg(1-\frac{q}{\kappa-1}\bigg)^{b-l}\,.
$$
The state space $\pml$
of the occupancy process 
is endowed with a partial order.
Let $o,o'\in\pml$,
we say that $o$ is lower than or equal to $o'$,
and we write $o\preceq o'$, if
$$\forall l\in\zl\qquad
o(0)+\cdots+o(l)\,\leq\,o'(0)+\cdots+o'(l)\,.$$

\section{Stochastic bounds}\label{Stobounds}
In this section we build simpler processes
in order to bound stochastically the occupancy process $\Ot$.
We will couple the simpler processes with the original occupancy process
and we will compare their invariant probability measures.

\subsection{Lower and upper processes}
We begin by constructing a lower process $\Otl$
and an upper process $\Otk$
in order to bound stochastically the original occupancy process $\Ot$.
In other words,
the lower and upper processes will be built so that
for every occupancy distribution $o\in\pml$,
if the three processes start from~$o$,
then
$$\forall n\geq 0\qquad
O^\ell_n\,\preceq\,
O_n\,\preceq\,
O^{K+1}_n\,.$$
The new processes will have simpler dynamics
than the original occupancy process.

Let us describe loosely the dynamics of the lower process.
As long as there are no master sequences present in the population,
the lower process evolves exactly as the original occupancy process.
As soon as a master sequence appears,
all the chromosomes in the Hamming classes $K+1,\dots,\ell$
are directly sent to the class $\ell$.
Moreover,
as long as the master sequence remains present in the population,
all mutations towards the classes $K+1,\dots,\ell$
are also sent to the Hamming class $\ell$.
The dynamics of the upper process is similar,
this time with the Hamming class $\ell$
replaced by the class $K+1$.
The rest of the section is devoted to formalising this construction.

Let $\Psi_O$ be the coupling map
defined in section~5.1 of~\cite{CerfWF}.
We modify this map in order to obtain a lower map $\Psi^\ell_O$
and an upper map $\Psi^{K+1}_O$.
The coupling map $\Psi_O$ takes two arguments,
an occupancy distribution $o\in\pml$
and a matrix $r\in\cR$,
where $\cR$ is the set of matrices of size $m\times(\ell+1)$
with coefficients in $[0,1]$.
The Markov chain $\Ot$ is built with the help of the map $\Psi_O$
and a sequence $(R_n)_{n\geq 1}$
of independent random matrices with values in $\cR$,
the entrances of the same random matrix $R_n$
being independent and identically distributed,
with uniform law over the interval $[0,1]$.

Let us define two maps 
$\pi_\ell,\,\pi_{K+1}:\pml\to\pml$
by setting,
for every $o\in\pml$,
\begin{align*}
\pi_\ell(o)\,&=\,
\big(o(0),\dots,o(K),0,\dots,0,m-o\big((0)+\cdots+o(K)\big)\big)\,,\\
\pi_{K+1}(o)\,&=\,
\big(o(0),\dots,o(K),m-\big(o(0)+\dots+o(K)\big),0,\dots,0\big)\,.
\end{align*}
Obviously,
$$\forall o\in\pml\qquad
\pi_\ell(o)\,\preceq\,
o\,\preceq\,
\pi_{K+1}(o)\,.$$
We denote by $\cW$
the set of occupancy distributions
having at least one master sequence, i.e.,
$$\cW\,=\,
\lbrace\,
o\in\pml:o(0)\geq 1\,\rbrace\,,$$
and we denote by $\cN$
the set of occupancy distributions
having no master sequences, i.e.,
$$\cN\,=\,
\lbrace\,
o\in\pml: o(0)=0\,\rbrace\,.$$
Let us define 
$$\olen\,=\,(1,0,\dots,0,m-1)\,,\qquad
\olex\,=\,(0,\dots,0,m)\,.$$
The occupancy distributions 
$\olen$ and $\olex$
are the absolute minima of the sets
$\cW$ and $\cN$.
We define the lower map $\Psi^\ell_O$ by setting,
for $o\in\pml$ and $r\in\cR$,
$$\Psi^\ell_O(o,r)\,=\,
\begin{cases}
\ \Psi_O(o,r) &\ \text{ if }\ o\in\cN\ \text{ and }\ \Psi_O(o,r)\not\in\cW\,,\\
\ \olen &\ \text{ if }\ o\in\cN\ \text{ and }\ \Psi_O(o,r)\in\cW\,,\\
\ \pi_\ell\big(\Psi_O(\pi_\ell(o),r)\big) &
\ \text{ if }\ o\in\cW\ \text{ and }\ \Psi_O(\pi_\ell(o),r)\not\in\cN\,,\\
\ \olex &\ \text{ if }\ o\in\cW\ \text{ and }\ \Psi_O(\pi_\ell(o),r)\in\cN\,.
\end{cases}$$
Likewise,
we define the occupancy distributions 
$$\oken\,=\,(m,0,\dots,0)\,,\qquad
\okex\,=\,(0,m,0,\dots,0)\,,$$
which are the absolute maxima of the sets 
$\cW$ and $\cN$.
We define an upper map $\Psi^{K+1}_O$ by setting,
for $o\in\pml$ and $r\in\cR$,
 $$\Psi^{K+1}_O(o,r)=
\begin{cases}
\Psi_O(o,r) &\text{if }o\in\cN\text{ and }\Psi_O(o,r)\not\in\cW,\\
\oken &\text{if }o\in\cN\text{ and }\Psi_O(o,r)\in\cW,\\
\pi_{K+1}\big(\Psi_O(\pi_{K+1}(o),r)\big) &
\text{if }o\in\cW\text{ and }\Psi_O(\pi_{K+1}(o),r)\not\in\cN,\\
\okex &\text{if }o\in\cW\text{ and }\Psi_O(\pi_{K+1}(o),r)\in\cN.
\end{cases}$$
The coupling map $\Psi_O$ is monotone
---lemma~5.5 of~\cite{CerfWF}---
i.e., for every pair of occupancy distributions $o,o'$
and for every $r\in\cR$,
$$o\preceq o'
\,\Lra\,
\Psi_O(o,r)\preceq\Psi_O(o',r)\,.$$
We deduce that the lower map $\Psi^\ell_O$
is below the coupling map $\Psi_O$
and the upper map $\Psi^{K+1}_O$
is above the coupling map $\Psi_O$,
i.e.,
$$\forall o\in\pml\quad
\forall r\in\cR\qquad
\Psi^\ell(o,r)\,\preceq
\Psi_O(o,r)\,\preceq\,
\Psi^{K+1}_O(o,r)\,.$$
We use the lower and upper maps,
along with the i.i.d. sequence of random matrices $(R_n)_{n\geq 0}$,
in order to build a lower occupancy process $\Otl$
and an upper occupancy process $\Otk$.
Let $o\in\pml$ be the starting point of the processes.
We set $O^\ell_0=O^{K+1}_0=o$ and
$$\forall n\geq 1\qquad
O^\ell_n\,=\,\Psi^\ell(O^\ell_{n-1},R_n)\,,\qquad
O^{K+1}_n\,=\,\Psi^{K+1}(O^{K+1}_{n-1},R_n)\,.$$
\begin{proposition}\label{domio}
Suppose that the processes
$\Ot$, $\Otl$, $\Otk$
start all from the same occupancy distribution $o$.
We have
$$\forall n\geq 0\qquad
O^\ell_n\,\preceq\,
O_n\,\preceq\,
O^{K+1}_n\,.$$
\end{proposition}
The proof is similar to the proof of proposition~8.1 in~\cite{CerfM}.

\subsection{Dynamics of the bounding processes}\label{Dynabound}
We study now the dynamics of the lower and upper processes in $\cW$.
Since the calculations are the same for both processes,
we take $\t$ to be either $K+1$ or $\ell$,
and we denote by $\Ott$ the corresponding process.
For the process $\Ott$,
the states in the set 
$$\cT^\t\,=\,
\lbrace\,
o\in\pml:
o(0)\geq 1
\text{ and }
o(0)+\cdots+o(K)+o(\t)<m
\,\rbrace\,,$$
are transient,
and the states in $\smash{\cN\cup(\cW\setminus\cT^\t)}$
form a recurrence class.
Let us take a look at the transition mechanism restricted to
$\smash{\cN\cup(\cW\setminus\cT^\t)}$.
Since
$$\cW\setminus\cT^\t\,=\,
\lbrace\,
o\in\pml:
o(0)\geq 1
\text{ and }
o(0)+\cdots+o(K)+o(\t)=m
\,\rbrace\,,$$
a state in $\smash{\cW\setminus\cT^\t}$
is totally determined 
by the occupancy numbers of the Hamming classes $0,\dots,K$;
whenever the process $\Ott$ 
starts form a state in $\smash{\cW\setminus\cT^\t}$,
the dynamics of
$
\big(
O^\t_n(0),\dots,O^\t_n(K)
\big)_{n\geq 0}$
is Markovian until the time of exit from $\smash{\cW\setminus\cT^\t}$.
Let us define the set 
$$\dD\,=\,
\lbrace\,
z\in\N^{K+1}:
z_0+\cdots+z_K\leq m
\,\rbrace\,.$$
We define the projection $\pi:\pml\to\dD$ 
by setting, for $o\in\pml$,
$$\pi(o)\,=\,(o(0),\dots,o(K))\,.$$
We denote by $\Ztt$ the Markov chain with state space
$\dD$
and transition matrix given by:
for $z,z'\in\dD$ and for any $n\geq 0$,
let $o$ be the unique element of
$\pml\setminus\cT^\t$
such that $\pi(o)=z$,

$\bullet$ if $z_0,z'_0\geq 1$,
$$
P(Z^\t_{n+1}=z'
\,|\,
Z^\t_n=z)\,=\,
P(\pi(O^\t_{n+1})=z'
\,|\,
O^\t_n=o)\,.$$

$\bullet$ if $z_0\geq 1$ and $z'_0=0$,
$$
P(Z^\t_{n+1}=z^\t_{\text{exit}}
\,|\,
Z^\t_n=z)\,=\,
\sum_{z':z'_0=0}
P(\pi(O^\t_{n+1})=z'
\,|\,
O^\t_n=o)\,,$$
where $z^\ell_{\text{exit}}=(0,\dots,0)$ and
$z^{K+1}_{\text{exit}}=(0,m,0,\dots,0)$.

$\bullet$ if $z=z^\t_{\text{exit}}$,
$$
P(Z^\t_{n+1}=z^\t_{\text{enter}}
\,|\,
Z^\t_n=z^\t_{\text{exit}})\,=\,
1\,,$$
where $z^\ell_{\text{enter}}=(1,0,\dots,0)$
and $z^{K+1}_{\text{enter}}=(m,0,\dots,0)$\,.

The remaining non--diagonal coefficients
of the transition matrix are null.
The diagonal coefficients are chosen so that the matrix is stochastic,
i.e., each row adds up to 1.
Let us denote by $p^\t(z,z')$ the above transition matrix
and let us compute its value 
for $z,z'\in\dD$ such that $z_0,z_0'\geq 1$.
We introduce some notation first.
For $d\geq 1$ and a vector $v\in\R^d$,
we denote by $|v|_1$ the $L^1$ norm of $v$:
$$|v|_1\,=\,|v_1|+\cdots+|v_d|\,.$$
For $d\geq 1$, a square matrix $M\in\R^{d^2}$,
and $i\in\lbrace\,1,\dots,d\,\rbrace$,
we denote by $M(i,\cdot)$ or $M_{i\lcdot}$ the $i$--th row of $M$,
and by $M(\cdot,i)$ or $M_{\lcdot i}$ the $i$--th column of $M$.
We also denote by $|M|_1$ the $L^1$ norm of $M$ in $\R^{d^2}$:
$$|M|_1\,=\,\sum_{i,j=1}^d |M_{ij}|\,.$$
We say that a vector $s\in\dD$ is compatible with another vector $z\in\dD$,
and we write $s\sim z$, if
$$z_i=0\,\Ra\,s_i=0\quad \text{for}\quad i\in\zk\
\qquad\text{and}\qquad 
|z|_1=m\,\Ra\,|s|_1=m\,.$$ 
We say that a matrix $b\in\N^{(K+1)^2}$
is compatible with the vectors $s,z'\in\dD$,
and we write $b\sim(s,z')$, if
$$\forall i\in\zk\qquad
|b(i,\cdot)|_1\,\leq\,s_i\qquad
\text{and}\qquad
|b(\cdot,i)|_1\,\leq\,z'_i\,.$$
Finally,
for $i\in\zk\cup\lbrace\t\rbrace$,
we define $M_H(i)$
to be the vector of $[0,1]^{K+1}$
given by
$$M_H(i)\,=\,\big(
M_H(i,0),\dots,M_H(i,K)
\big)\,.$$
Let $z,z'\in\dD$ such that $z_0,z'_0\geq 1$.
We now use the transition mechanism of $\Ott$
in order to compute the value of $p^\t(z,z')$:
$$p^\t(z,z')\,=\
\sum_{s\sim z}
\sum_{b\sim(s,z')}
p^\t(z,s,b,z')\,,$$
where $p^\t(z,s,b,z')$
is the probability that,
given $Z^\t_n=z$:

$\bullet$ for $i\in\zk$, $s_i$ individuals from the class $i$ are selected,
and $m-|s|_1$ individuals from the class $\t$ are selected.
The probability of this event is 
$$\frac{m!}{s_0!\cdots s_K!(m-|s|_1)!}\times
\frac{(\s z_0)^{s_0}z_1^{s_1}\cdots z_K^{s_K}(m-|z|_1)^{m-|s|_1}}
{((\s-1)z_0+m)^m}\,,$$
$\bullet$ for $i,j\in\zk$,
$b_{ij}$ individuals from the class $i$ mutate to the class $j$,
and $s_i-|b(i,\cdot)|_1$ individuals from the class $i$ mutate 
to the class $\t$.
For $i\in\zk$, the probability of this event is
$$\frac{s_i!}{b_{i0}!\cdots b_{iK}!(s_i-|b(i,\cdot)|_1)!}\times
M_H(i,0)^{b_{i0}}\cdots M_H(i,K)^{b_{iK}}(1-|M_H(i)|_1)^{s_i-|b(i,\cdot)|_1}\,,
$$
$\bullet$ for $j\in\zk$, 
$z'_j-|b(\cdot,j)|_1$ individuals from the class $\t$ mutate to the class $j$,
and $m-|s|_1-|z'|_1+|b|_1$ individuals from the class $\t$ do not mutate to any of the classes $\zk$.
The probability of this event is
\begin{multline*}
\frac{(m-|s|_1)!}
{(z'_0-|b(\cdot,0)|_1)!\cdots(z'_K-|b(\cdot,K)|_1)!(m-|s|_1-|z'|_1+|b|_1)!}
\\[4 pt]
\times M_H(\t,0)^{z'_0-|b(\cdot,0)|_1}\cdots M_H(\t,K)^{z'_K-|b(\cdot,K)|_1}
(1-|M_H(\t)|)^{m-|s|_1-|z'|_1+|b|_1}\,.
\end{multline*}
Finally,
\begin{multline*}
p^\t(z,s,b,z')\,=\,
\frac{m!}{s_0!\cdots s_K!(m-|s|_1)!}\times
\frac{(\s z_0)^{s_0}z_1^{s_1}\cdots z_K^{s_K}(m-|z|_1)^{m-|s|_1}}
{((\s-1)z_0+m)^m}\times\\
\prod_{i=0}^K 
\frac{s_i!}{b_{i0}!\cdots b_{iK}!(s_i-|b_{i\lcdot}|_1)!}\times
M_H(i,0)^{b_{i0}}\cdots M_H(i,K)^{b_{iK}}(1-|M_H(i)|_1)^{s_i-|b_{i\lcdot}|_1}\\
\times
\frac{(m-|s|_1)!}
{(z'_0-|b_{\lcdot 0}|_1)!\cdots(z'_K-|b_{\lcdot K}|_1)!(m-|s|_1-|z'|_1+|b|_1)!}
\\[4 pt]
\times M_H(\t,0)^{z'_0-|b_{\lcdot 0}|_1}\cdots M_H(\t,K)^{z'_K-|b_{\lcdot K}|_1}
(1-|M_H(\t)|_1)^{m-|s|_1-|z'|_1+|b|_1}\,.
\end{multline*}

\subsection{Bounds on the invariant measure}\label{Invmes}
Let us denote by
$\mu_O,\mu_O^\ell,\mu_O^{K+1}$
the invariant probability measures
of the processes
$\Ot,\Otl,\Otk$.
Let $\nu$
be the image measure of $\mu_O$
through the map
$$o\in\pml\longmapsto \frac{o(0)+\cdots+o(K)}{m}\,=\,
\frac{|\pi(o)|_1}{m}\in[0,1]\,.$$
For every function $g:[0,1]\mapsto\R$,
$$
\int_{[0,1]}g\,d\nu\,=\,
\int_{\pml} g\bigg(\frac{|\pi(o)|_1}{m}\bigg)\,d\mu_O\,=\,
\lim_{n\to\infty}E\bigg(
g\bigg(\frac{|\pi(O_n)|_1}{m}\bigg)
\bigg)\,.
$$
Let now $g:[0,1]\mapsto\R$
be an increasing function such that $g(0)=0$.
Thanks to proposition~\ref{domio},
the following inequalities hold:
for all $n\geq 0$,
$$
g\bigg(
\frac{|\pi(O_n^\ell)|_1}{m}
\bigg)\,
\leq\, g\bigg(
\frac{|\pi(O_n)|_1}{m}
\bigg)\,\leq\,
g\bigg(
\frac{|\pi(O_n^{K+1})|_1}{m}
\bigg)\,.
$$
Taking the expectation and sending $n$ to $\infty$
we deduce that
$$
\int_{\pml}g\bigg(
\frac{|\pi(o)|_1}{m}\bigg)
\,d\mu_O^\ell(o)
\,\leq\,\int_{[0,1]}g\,d\nu\,\leq\,
\int_{\pml}g\bigg(
\frac{|\pi(o)|_1}{m}\bigg)
\,d\mu_O^{K+1}(o)\,.
$$
Next,
we seek to estimate the above integrals.
The strategy is the same for the lower and upper integrals;
we set $\t$ to be either $K+1$ or $\ell$
and we study the invariant probability measure $\mu_O^\t$.
We will rely on the following renewal result.
Let $\cE$ be a finite set
and let $(X_n)_{n\geq 0}$
be an ergodic Markov chain with state space $\cE$
and invariant probability measure $\mu$.
Let $\cW$ be a subset of $\cE$
and let $e\in\cE$ be a state outside $\cW$.
We define
$$\tau^*\,=\,\inf\lbrace\,
n\geq 0 : X_n\in\cW
\,\rbrace\,,\qquad
\tau\,=\,\inf\lbrace\,
n\geq \tau^*: X_n=e
\,\rbrace\,.$$
\begin{proposition}\label{renewal}
For every function $f:\cE\mapsto\R$,
we have
$$\int_\cE f\,d\mu\,=\,
\frac{\displaystyle E\Bigg(
\sum_{n=0}^{\tau-1} f(X_n)
\,\Bigg|\,
X_0=e
\Bigg)}{E(\tau\,|\,X_0=e)}\,.$$
\end{proposition}
The proof is standard and similar to that of proposition~9.2 of \cite{CerfM}.
We apply the renewal result to the process $\Ott$
restricted to $\smash{\cN\cup(\cW\setminus\cT^\t)}$,
the set $\smash{\cW\setminus\cT^\t}$,
the occupancy distribution $\otex$
and the function $o\mapsto g\big(|\pi(o)|_1/m\big)$.
We set
$$\tau^*\,=\,\inf\lbrace\,
n\geq 0 : O^\t_n\in\cW
\,\rbrace\,,\qquad
\tau\,=\,\inf\lbrace\,
n\geq \tau^* : O^\t_n=\otex
\,\rbrace\,.$$
Applying the renewal theorem we get
$$
\int_{\pml}g\bigg(
\frac{|\pi(o)|_1}{m}
\bigg)\,d\mu_O^\t(o)\,=\,
\frac{\displaystyle
E\Bigg(
\sum_{n=0}^{\tau-1}
g\bigg(
\frac{|\pi(O^\t_n)|_1}{m}
\bigg)
\,\Bigg|\,
O^\t_0=\otex
\Bigg)}{E(\tau\,|\, O^\t_0=\otex)}\,.
$$
Whenever the process $\Ott$ is in $\smash{\cW\setminus\cT^\t}$,
the dynamics of the first $K+1$ Hamming classes,
$\smash{\big(
\pi(O^\t_n)
\big)_{n\geq 0}}$,
is that of the Markov chain $\Ztt$
defined at the end of the previous section.
Let us suppose that $\Ztt$ starts from $z^\t_{\text{enter}}\in\dD$,
where $z^\ell_{\text{enter}}=(1,0,\dots,0)$ and 
$z^{K+1}_{\text{enter}}=(m,0,\dots,0)$.
Let $\tau_0$ be the first time that $Z_n^\t(0)$ becomes null:
$$\tau_0\,=\,\inf\lbrace\,
n\geq 0 : Z_n^\t(0)=0
\,\rbrace\,.$$
Since the process $\Ott$ always enters the set $\smash{\cW\setminus\cT^\t}$
at the state $\oten$,
the law of $\tau_0$ is the same as the law of $\tau-\tau^*$
for the process $\Ott$ starting from $\otex$.
We conclude that the trajectories
$\big(
\pi(O^\t_n)
\big)_{\tau^*\leq n\leq \tau}$ 
and $\big( Z^\t_n \big)_{0\leq n\leq \tau_0}$
have the same law.
Therefore,
\begin{align*}
E(\tau-\tau^*
\,|\,
O^\t_0=\otex)\,&=\,
E(\tau_0\,|\, Z^\t_0=z^\t_{\text{enter}})\,,\\
E\Bigg(
\sum_{n=\tau^*}^{\tau-1}
g\bigg(
\frac{|\pi(O_n^\t)|_1}{m}
\bigg)
\,\Bigg|\,
O^\t_0=\otex
\Bigg)\,&=\,
E\Bigg(
\sum_{n=0}^{\tau_0-1}
g\bigg(
\frac{|Z_n^\t|_1}{m}
\bigg)
\,\Bigg|\,
Z^\t_0=z^\t_{\text{enter}}
\Bigg)\,.\hspace*{20 pt}
\end{align*}
Thus, 
we can rewrite the formula for the invariant probability measure $\mu_O^\t$ as follows:
\begin{multline*}
\int_{\pml}g\bigg(\frac{|\pi(o)|_1}{m}\bigg)
\,d\mu_O^\t(o)\,=\,
\frac{\displaystyle 
E\Bigg(
\sum_{n=0}^{\tau^*-1}
g\bigg(
\frac{|\pi(O^\t_n)|_1}{m}
\bigg)
\,\Bigg|\,
O^\t_0=\otex
\Bigg)}{E(\tau^*\,|\, O^\t_0=\otex)
+E(\tau_0\,|\, Z^\t_0=z^\t_{\text{enter}})}
\\
+\frac{\displaystyle
E\Bigg(
\sum_{n=0}^{\tau_0-1}
g\bigg(
\frac{|Z^\t_n|_1}{m}
\bigg)
\,\Bigg|\,
Z^\t_0=z^\t_{\text{enter}}
\Bigg)}{E(\tau^*\,|\, O^\t_0=\otex)
+E(\tau_0\,|\, Z^\t_0=z^\t_{\text{enter}})}\,.\hfil
\end{multline*}
The objective of the following sections
is to estimate each of the terms 
appearing in the right hand side of this formula.

\section{Replicating Markov chains}
We study now the Markov chains 
$\Ztl$ and $\Ztk$.
The computations are the same for both processes,
we take $\t$ to be either $K+1$ or $\ell$
and we study the Markov chain $\Ztt$.
We will carry out all of our estimates in the asymptotic regime
$$\ell\to +\infty\,,\qquad
m\to +\infty\,,\qquad
q\to 0\,,\qquad
\ell q\to a\in\,]0,+\infty[\,\,.$$
We will say that a property holds asymptotically,
if it holds for $\ell,m$ large enough,
$q$ small enough and $\ell q$ close enough to $a$.

\subsection{Large deviations for the transition matrix}\label{LDtrans}
We define the set
$\cD\subset \R^{K+1}$ 
by
$$\cD\,=\,
\big\lbrace\,
r\in\R^{K+1}:
r_0\geq 0,\dots,r_K\geq 0 \text{ and } r_0+\cdots+r_K\leq 1
\,\big\rbrace\,.$$
For
$p,t\in\cD$, 
we define the quantity
$I_K(p,t)$
as follows:
%
$$I_K(p,t)\,=\,
\sum_{k=0}^K t_k\ln\frac{t_k}{p_k}+
(1-|t|_1)
\ln\frac{1-|t|_1}{1-|p|_1}\,,$$
%
%
We make the convention that 
$a\ln(a/b)=0$ if $a=b=0$.
The function
$I_K(p,\cdot)$
is the rate function
governing the large deviations of a multinomial distribution
with parameters $n$ and $p_0,\dots,p_K,1-|p|_1$.
We have the following estimate for the multinomial coefficients:

\begin{lemma}
For all
$n\geq1$,
$N<n$
and
$i_0,\dots,i_N\in \lbrace 0,\dots,n\rbrace$
such that
$s=i_0+\cdots+i_N\leq n$,
we have
$$\Bigg|
\ln\frac{n!}{i_0!\cdots i_N! (n-s)!}
+\sum_{k=0}^N i_k\ln\frac{i_k}{n}
+(n-s)\ln\frac{n-s}{n}
\Bigg|\,\leq\,
(N+2)\ln n + 2N +3\,.$$
\end{lemma}
The proof is similar to that of lemma~7.1 of~\cite{CerfWF}.

We define a function
$f:\cD\to\cD$
by setting
$$\forall r\in\cD\qquad
f(r)\,=\,
\frac{1}{(\s-1)r_0+1}(\s r_0,r_1,\dots,r_K)\,.$$
We also define a function
$I_\ell:\cD\times\cD\times[0,1]^{(K+1)^2}\times\cD\to[0,+\infty]$
by setting, for $r,\xi,t\in\cD$ and $\b\in[0,1]^{(K+1)^2}$,
\begin{multline*}
I_\ell(r,\xi,\b,t)\,=\,
I_K(f(r),\xi)+
\sum_{k=0}^K \xi_k I_K\Big(
M_H(k),\xi_k^{-1}\b(k,\cdot)
\Big)\\+\,
(1-|\xi|_1)
I_K\Big(
M_H(\t),
(1-|\xi|_1)^{-1}
(t_0-|\b(\cdot,0)|_1,\dots,t_K-|\b(\cdot,K)|_1)
\Big)\,.
\end{multline*}
Thanks to the previous identities,
for all 
$z,z',s\in\dD$
and
${b\in\N^{(K+1)^2}}$,
we can express the logarithm of the transition probability
$p^\t(z,s,b,z')$ as follows:
\begin{multline*}
\ln p^\t(z,s,b,z')\,=\,
-mI_K\Big(f\Big(\frac{z}{m}\Big),\frac{s}{m}\Big)-
\sum_{k=0}^K s_k I_K\big(M_H(k),s_k^{-1}b(k,\cdot)\big)\\
-(m-|s|_1)I_K\Big( 
M_H(\t),
(m-|s|_1)^{-1}\big(z'_0-|b(\cdot,0)|_1,\dots,z'_K-|b(\cdot,K)|_1\big)
\Big)\\
+\Phi(z,s,b,z')\,=\,
-m I_\ell\bigg( \frac{z}{m},\frac{s}{m},\frac{b}{m},\frac{z'}{m} \bigg)
+\Phi(z,s,b,z')\,.
\end{multline*}
The error term
$\Phi(z,s,b,z')$
satisfies, for all $m\geq 1$,
$$\forall z,z',s\in\dD\quad \forall b\in\N^{(K+1)^2}\qquad
\big|\Phi(z,s,b,z')\big|\,\leq\, C(K)(\ln m+1)\,,$$
where
$C(K)$
is a constant that depends on $K$ but not on $m$.
In the asymptotic regime,
for all
$i,j\geq 0$,
$$M_H(i,j)\,\lra\,M_\infty(i,j)\,=\,\begin{cases}
\quad \exa\displaystyle\frac{a^{j-i}}{(j-i)!}& \quad \text{si }\, i\leq j\,,\\
\quad 0 & \quad \text{si }\, i>j\,.
\end{cases}$$
For
$k\in\zk
$, we set
$$M_\infty(k)\,=\,\big( M_\infty(k,0),\dots,M_\infty(k,K) \big)\,.$$
For 
$t\in \cD$, 
we call
$\cB(t)$
the subset of
$[0,1]^{(K+1)^2}$
of the upper triangular matrices
$\b$ 
such that
the sum of the columns of $\b$ is equal to the vector $t$,
i.e.,
\begin{multline*}
\hspace*{40 pt}\cB(t)\,=\,
\big\lbrace\,
\b\in[0,1]^{(K+1)^2}:
\b_{ij}=0 \text{ for } i>j \ \text{ and }\\
|\b(\cdot,k)|_1=t_k \text{ for } 0\leq k\leq K
\,\big\rbrace\,.\hspace*{40 pt}
\end{multline*}
In the asymptotic regime, for
$r,\xi,t\in\cD$
and
$\b\in [0,1]^{(K+1)^2}$,
we get
$$I_\ell(r,\xi,\b,t)\,\lra\,\begin{cases}
\quad I(r,\xi,\b,t)& \quad \text{if }\, \b\in\cB(t)\,,\\
\quad +\infty& \quad \text{otherwise,}
\end{cases}$$
where the function
$I(r,\xi,\b,t)$
is given by
$$I(r,\xi,\b,t)\,=\,
I_K(f(r),\xi)+
\sum_{k=0}^K \xi_k I_K(M_\infty(k),\xi_k^{-1}\b(k,\cdot))\,.$$
We define a function
$V_1:\cD\times\cD\to [0,\infty]$
by setting, for
$r,t\in\cD$,
$$V_1(r,t)\,=\,
\inf\big\lbrace\,
I(r,\xi,\b,t):
\xi\in\cD,\ \b\in\cB(t)
\,\big\rbrace\,.$$ 
For $r\in\R^{K+1}$,
we denote by $\lfloor r \rfloor$
the vector 
$\lfloor r\rfloor=(\lfloor r_0\rfloor,\dots,\lfloor r_K\rfloor)$.
\begin{proposition}\label{pgdtrans1}
The one step transition probabilities of the Markov chain  
$\Ztt$
verify the large deviations principle governed by
$V_1$:

$\bullet$ For any subset $U$ of $\cD$ and for any $\rho\in\cD$,
we have, for $n\geq 0$,
$$
-\inf\big\lbrace\,
V_1(\rho,t):t\in \uro
\,\big\rbrace\,
\leq\,\liminf_{\genfrac{}{}{0pt}{1}{\ell,m\to\infty,\,q\to 0}{{\ell q} \to a}}
\frac{1}{m}\ln P\big( Z^\t_{n+1}\in mU \,\big|\, Z^\t_n=\lfloor m\rho  \rfloor \big)\,.
$$

$\bullet$ For any subsets $U,U'$ of $\cD$,
we have, for $n\geq 0$,
\begin{multline*}
\hspace*{30 pt}\limsup_{\genfrac{}{}{0pt}{1}{\ell,m\to\infty,\,q\to 0}{{\ell q} \to a}}
\frac{1}{m}\ln \sup_{z\in mU} P\big( Z^\t_{n+1}\in mU' \,\big|\, Z^\t_n=z \big)
\\\leq\,
-\inf\big\lbrace\,
V_1(r,t):
r\in \overline{U},\,
t\in \overline{U}'
\,\big\rbrace\,.\hspace*{40 pt}
\end{multline*}
\end{proposition}
\begin{proof}
We begin by showing the large deviations upper bound.
Let $U,U'$
be two subsets of
$\cD$
and take $z\in mU$.
For 
$n\geq 0$,
\begin{align*}
P\big(Z^\t_{n+1}\in mU'
\,|\, Z^\t_n=z\big)\,&=\,
\sum_{z'\in mU'\cap \dD} p^\t(z,z')\\
&=\,\sum_{z'\in mU'\cap\dD}\,
\sum_{s\sim z}\,
\sum_{b\sim(s,z')} p^\t(z,s,b,z')\,.
\end{align*}
Thanks to the estimates on $p^\t$,
we have,
for $m\geq 1$,
\begin{multline*}
\sup_{z\in mU}P\big(Z^\t_{n+1}\in mU'
\,|\, Z^\t_n=z\big)
\\\leq\,
(m+1)^{C(K)}\max\big\{\,
p^\t(z,s,b,z'):
z\in mU,\
s\sim z,\ 
z'\in mU',\ 
b\sim (s,z')
\,\big\}
\\\leq\,
(m+1)^{C(K)}\exp\bigg(
-m\min\bigg\lbrace\,
I_\ell\Big(
\frac{z}{m},
\frac{s}{m},
\frac{b}{m},
\frac{z}{m}
\Big):\,
\begin{matrix}
z\in mU,\
z'\in mU'\\
s\sim z,\
b\sim (s,z')
\end{matrix}
\,\bigg\rbrace
\bigg),
\end{multline*}
where
$C(K)$ 
is a constant depending on
$K$ 
but not on 
$m$.
For each
$m\geq 1$,
let
$z_m,s_m,z'_m\in\dD$,
$b_m\in\zm^{(K+1)^2}$ 
be four terms that realise the above minimum.
We observe next the expression
$$\limsup_{\genfrac{}{}{0pt}{1}{\ell,m\to\infty,\,q\to 0}{{\ell q} \to a}}
-I_\ell\bigg(
\frac{z_m}{m},
\frac{s_m}{m},
\frac{b_m}{m},
\frac{z'_m}{m}
\bigg)\,.$$
Since
$\cD$
and
$[0,1]^{(K+1)^2}$
are compact sets,
up to the extraction of a subsequence,
we can suppose that when
$m\to\infty$,
$$\frac{z_m}{m}\to\rho\in\overline{U}\,,\ \quad
\frac{s_m}{m}\to\xi\in\cD\,,\ \quad
\frac{b_m}{m}\to \b\in[0,1]^{(K+1)^2}\,,\ \quad
\frac{z'_m}{m}\to t\in\overline{U}'\,.$$
If 
$\b$ 
is not an upper triangular matrix,
or if, for some $j\in\zk$, 
$|\b(\cdot,j)|\neq t_j$,
the limit is $-\infty$.
Thus, the only case we need to take care of is when
$\b\in\cB(t)$.
In this case, we have
$$
\limsup_{\genfrac{}{}{0pt}{1}{\ell,m\to\infty,\,q\to 0}{{\ell q} \to a}} 
-I_\ell\Big(
\frac{z_m}{m},
\frac{s_m}{m},
\frac{b_m}{m},
\frac{z_m}{m}
\Big)
\,\leq\,
-I(\rho,\xi,\b,t)\,.
$$
Optimising with respect to $\rho,\xi,\b,t$,
we obtain the upper bound of the large deviations principle.

We show next the lower bound.
Let
$\xi,t\in\cD$
and
$\b\in\cB(t)$.
We have
\begin{multline*}
P\big(
Z^\t_{n+1}=\lfloor mt \rfloor 
\,\big|\,
Z^\t_n=\lfloor m\rho \rfloor
\big)
\,\geq\,
p^\t(\lfloor m\rho \rfloor,
\lfloor m\xi \rfloor,
\lfloor m\b \rfloor,
\lfloor mt \rfloor)
\\\geq\,
(m+1)^{-C(K)}
\exp\bigg(
-mI_\ell\bigg(
\frac{\lfloor m\rho \rfloor}{m},
\frac{\lfloor m\xi \rfloor}{m},
\frac{\lfloor m\b \rfloor}{m},
\frac{\lfloor mt \rfloor}{m}
\bigg)
\bigg)\,.
\end{multline*}
We take the logarithm and we send
$m,\ell$ to $\infty$ and $q$ to $0$.
We obtain then
$$
\liminf_{\genfrac{}{}{0pt}{1}{\ell,m\to\infty,\,q\to 0}{{\ell q} \to a}}
\frac{1}{m}\ln P\big(
Z^\t_{n+1}=\lfloor tm \rfloor
\,\big|\,
Z^\t_n=\lfloor \rho m \rfloor
\big)
\,\geq\,
-I(\rho,\xi,\b,t)\,.
$$
Moreover, if
$t\in\uro$,
for $m$ large enough,
$\lfloor tm \rfloor$ belongs to $mU$.
Therefore,
$$
\liminf_{\genfrac{}{}{0pt}{1}{\ell,m\to\infty,\,q\to 0}{{\ell q} \to a}}
\frac{1}{m}\ln P\big(
Z^\t_{n+1}\in mU
\,\big|\,
Z^\t_n=\lfloor m\rho \rfloor
\big)
\,\geq\,
-I(\rho,\xi,\b,t)\,.
$$
We optimise over 
$\xi,\b,t$ 
and we obtain the large deviations lower bound.
\end{proof}
A similar proof shows that the $l$--step transition probabilities of $\Ztt$
also satisfy a large deviations principle.
For $l\geq 1$,
we define a function 
$V_l$ on $\cD\times\cD$ as follows:
\begin{multline*}
V_l(r,t)
\,=\,
\inf\Big\lbrace\,
\sum_{k=0}^{l-1} I(\rho_k,\xi_k,\b_k,\rho_{k+1}):\\
\rho_0=r,\ \rho_l=t,\ 
\rho_k, \xi_k\in\cD,\ 
\b_k\in\cB(t)
\ \text{ for }\ 0\leq k<l
\,\Big\rbrace\,.
\end{multline*}
\begin{corollary}\label{pgdtransl}
For 
$l\geq 1$,
the $l$--step transition probabilities of
$(Z^\t_n)_{n\geq 0}$
satisfy the large deviations principle governed by
$V_l$:

$\bullet$ For any subset $U$ of $\cD$ and for any $\rho\in\cD$,
we have, for $n\geq 0$,
$$
-\inf\big\lbrace\,
V_l(\rho,t):t\in \uro
\,\big\rbrace\,
\,\leq\,\liminf_{\genfrac{}{}{0pt}{1}{\ell,m\to\infty,\,q\to 0}{{\ell q} \to a}}
\frac{1}{m}\ln P\big( Z^\t_{n+l}\in mU \,\big|\, Z^\t_n=\lfloor \rho m \rfloor \big)\,.\hfil
$$

$\bullet$ For any subsets $U,U'$ of $\cD$,
we have, for $n\geq 0$,
\begin{multline*}
\hspace*{30 pt}\limsup_{\genfrac{}{}{0pt}{1}{\ell,m\to\infty,\,q\to 0}{{\ell q} \to a}}
\frac{1}{m}\ln \sup_{z\in mU} P\big( Z^\t_{n+l}\in mU' \,\big|\, Z^\t_n=z \big)
\\
\hspace*{30 pt}\leq\,
-\inf\big\lbrace\,
V_l(r,t):
r\in \overline{U},\,
t\in \overline{U}'
\,\big\rbrace\,.\hspace*{40 pt}
\end{multline*}
\end{corollary}

\subsection{Perturbed dynamical system}
We look next for the zeros of the rate function
$I(r,\xi,\b,t)$.
We see that 
$I(r,\xi,\b,t)=0$
if and only if
$\xi=f(r)$,
$\b\in\cB(t)$
and
$\b(k,\cdot)/\xi_k=M_\infty(k)$
for $0\leq k\leq K$.
We define a function
$F=(F_0,\dots,F_K):\cD\to\cD$
by setting, for $r\in\cD$ and $k\in\zk$,
$$F_k(r)\,=\,\sum_{i=0}^k f_i(r)\exa\frac{a^{k-i}}{(k-i)!}\,.$$
Replacing $f$ by its value in the above formula, we can rewrite,
for $0\leq k\leq K$,
$$F_k(r)\,=\,
\frac{\exa}{(\s-1)r_0+1}\bigg(
\frac{a^k}{k!}\s r_0+\sum_{i=1}^k\frac{a^{k-i}}{(k-i)!}r_i 
\bigg)\,.$$
The Markov chain
$\Ztt$
can be seen as a perturbation of the dynamical system
associated to the map $F$:
$$z^0\in\cD\,,\qquad \forall n\geq 1\quad z^n\,=\,F(z^{n-1})\,.$$
Let $\rho^*$ be the point of $\cD$ given by:
$$\forall k\in\zk\qquad 
\rho^*_k\,=\,(\s\exa-1)\frac{a^k}{k!}\sum_{i\geq 1}\frac{i^k}{\s^i}\,.$$
\begin{proposition}\label{convsd}
We have the following dichotomy:

$\bullet$ if $\s\exa\leq 1$,
the function $F$ has a single fixed point,
$0$,
and
$(z^n)_{n\geq 0}$
converges to $0$.

$\bullet$ if $\s\exa>1$,
the function $F$ has two fixed points,
$0$ and $\rho^*$.
If $z^0_0=0$, 
the sequence 
$(z^n)_{n\in\N}$
converges to 0,
whereas if $z^0_0>0$,
the sequence 
$(z^n)_{n\in\N}$
converges to $\rho^*$.
\end{proposition}
\begin{proof}
For ${k\in\zk}$,
the function 
$F_k(r)$
is a function of $r_0,\dots,r_k$ only;
we can inductively solve the system of equations
$$F_k(r)\,=\,r_k\,,\qquad 0\leq k\leq K\,.$$
For $k=0$, we have
$$F_0(r)\,=\,\frac{\s\exa r_0}{(\s-1)r_0+1}\,.$$
The equation $F_0(r)=r_0$ has two solutions:
$r_0=0$ and $r_0=\rho^*_0$.
For $k$ in $\lbrace\, 1,\dots,K\,\rbrace$, we have
$F_k(r)=r_k$
if and only if
$$r_k\,=\,\frac{\exa}{(\s-1)r_0+1-\exa}
\bigg(
\frac{a^k}{k!}\s r_0+\sum_{i=1}^{k-1}\frac{a^{k-i}}{(k-i)!}r_i
\bigg)\,.$$
We end up with a recurrence relation.
If the initial condition is
$r_0=0$,
the only solution is
$r_k=0$ for all $k\in\zk$,
whereas if the initial condition is
$r_0=\rho^*_0$,
the only solution is
$r_k=\rho^*_k$ for all $k\in\zk$,
this last assertion is shown in section~2.2 of~\cite{CD}.

It remains to show the convergence.
We will show the convergence in the case
$\s\exa>1$, $z^0_0>0$.
The other cases are dealt with in a similar fashion, or are even simpler.
We will prove the convergence by induction on the coordinates.
Since the function
$$F_0(r)=\frac{\s\exa r_0}{(\s-1)r_0+1}$$
is increasing, concave, 
and satisfies $F_0(\rho^*)=\rho^*_0$,
the sequence
$(z^n_0)_{n\geq 0}$
is monotone and converges to 
$\rho^*_0$.
Let $k\in\lbrace\,1,\dots,K\,\rbrace$
and let us suppose that
the following limit holds: 
$$\lim_{n\to\infty}(z^n_0,\dots,z^n_{k-1})\,=\,(\rho^*_0,\dots,\rho^*_{k-1})\,.$$
Let $\e>0$.
We define two functions
$\underline{F},\overline{F}:[0,1]\to[0,1]$
by setting, for 
$\rho\in[0,1]$,
\begin{align*}
\underline{F}(\rho)\,=\,
\frac{\exa}{(\s-1)(\rho^*_0+\e)+1}\bigg(
\frac{a^k}{k!}\s(\rho^*_0-\e)
+\sum_{i=1}^{k-1}\frac{a^{k-i}}{(k-i)!}(\rho^*_i-\e)+\rho
\bigg)\,,\\
\overline{F}(\rho)\,=\,
\frac{\exa}{(\s-1)(\rho^*_0-\e)+1}\bigg(
\frac{a^k}{k!}\s(\rho^*_0+\e)
+\sum_{i=1}^{k-1}\frac{a^{k-i}}{(k-i)!}(\rho^*_i+\e)+\rho
\bigg)\,.
\end{align*}
By the induction hypothesis,
there exists $N\in\N$ such that for all $n\geq N$
and $i\in\lbrace\, 0,\dots,k-1\,\rbrace$,
$|z^n_i-\rho^*_i|<\e$.
We have then, for all $n\geq N$
and for all $\rho\in[0,1]$,
$$\underline{F}(\rho)\,\leq\,
F_k(z^n_0,\dots,z^n_{k-1},\rho)\,\leq\,
\overline{F}(\rho)\,.$$
We define two sequences,
$(\underline{z}^n)_{n\geq N}$
and 
$(\overline{z}^n)_{n\geq N}$,
by setting
${\underline{z}^N=\overline{z}^N=z^N_k}$
and for $n>N$
$$\underline{z}^n\,=\,\underline{F}(\underline{z}^{n-1})\,,\qquad
\overline{z}^n\,=\,\overline{F}(\overline{z}^{n-1})\,.$$
Thus, for all $n\geq N$, we have
$\underline{z}^n\leq
z^n_k\leq
\overline{z}^n$.
Since $\underline{F}(\rho)$ and 
$\overline{F}(\rho)$
are affine functions,
and for $\e$ small enough
their main coefficient is strictly smaller than $1$,
the sequences $(\underline{z}^n)_{n\geq N}$
and $(\overline{z}^n)_{n\geq N}$
converge to the fixed points of the functions
$\underline{F}$ et $\overline{F}$,
which are given by:
\begin{align*}
\underline{\rho}^*_k\,=\,
\frac{\exa}{(\s-1)(\rho^*_0+\e)+1-\exa}\bigg(
\frac{a^k}{k!}\s(\rho^*_0-\e)
+\sum_{i=1}^{k-1}\frac{a^{k-i}}{(k-i)!}(\rho^*_i-\e)
\bigg)\,,\\
\overline{\rho}^*_k\,=\,
\frac{\exa}{(\s-1)(\rho^*_0-\e)+1-\exa}\bigg(
\frac{a^k}{k!}\s(\rho^*_0+\e)
+\sum_{i=1}^{k-1}\frac{a^{k-i}}{(k-i)!}(\rho^*_i+\e)
\bigg)\,.
\end{align*}
We let $\e$ go to $0$ and we see that
$$\lim_{n\to\infty}z^n_k\,=\,
\frac{\exa}{(\s-1)\rho^*_0+1-\exa}\bigg(
\frac{a^k}{k!}\s\rho^*_0
+\sum_{i=1}^{k-1}\frac{a^{k-i}}{(k-i)!}\rho^*_i
\bigg)\,=\,\rho^*_k\,,$$
which finishes the inductive step.
\end{proof}

\subsection{Comparison with the master sequence}\label{Compms}
In section~\ref{Stobounds}, in order to build the bounding occupancy processes,
we have fixed an integer $K\geq 0$
and we have kept the relevant information 
about the dynamics of the occupancy numbers of the Hamming classes $0,\dots,K$.
Let us call $\Ttl$ and $\Ttk$ the lower and upper occupancy processes
that are obtained for $K=0$,
and let us call,
as before,
$\Otl$ and $\Otk$
the lower and upper occupancy processes corresponding to $K>0$.
Let us define the following stopping times:
$$\tau(\T^\ell)\,=\,\inf\lbrace\,
n\geq 0 : \T^\ell_n\in\cN
\,\rbrace\,,\qquad
\tau(O^{K+1})\,=\,\inf\lbrace\,
n\geq 0 : O^{K+1}_n\in\cN
\,\rbrace\,.$$
We have constructed the processes $\Ttl,\Ttk,\Otl,\Otk$ 
in such a way that
they are all coupled and the following relations hold:
if the four processes start from the same occupancy distribution
$o\in\cW$, then
\begin{align*}
\forall n\in \lbrace\, 0,\dots,\tau(\T^\ell)\,\rbrace\qquad
&\T^\ell_n\,\preceq\,
O^\ell_n\,\preceq\,
O^{K+1}_n\,\preceq\,
\T^{K+1}_n\,,\\
\forall n\in \lbrace\, 0,\dots, \tau(O^{K+1})\,\rbrace\qquad
& O^{K+1}_n\,\preceq\,
\T^{K+1}_n\,.
\end{align*}
These inequalities are naturally inherited 
by the Markov chains derived from the occupancy processes;
let $\Ztl$ and $\Ztk$
be the Markov chains associated to the processes 
$\Otl$ and $\Otk$, as in the end of section~\ref{Dynabound}.
Likewise,
let $(Y^\ell_n)_{n\geq 0}$ and $(Y^1_n)_{n\geq 0}$
be the Markov chains associated to the processes $\Ttl$ and $\Ttk$.
The state space of the Markov chains $\Ztl$, $\Ztk$ 
is the set $\dD$ defined in section~\ref{Dynabound},
whereas the state space for the Markov chains 
$(Y^\ell_n)_{n\geq 0}$, $(Y^1_n)_{n\geq 0}$
is $\zm$.
Let us define the following stopping times:
$$\tau(Y^\ell)\,=\,\inf\lbrace\,
n\geq 0: Y^\ell_n=0
\,\rbrace\,,\qquad
\tau(Z^{K+1})\,=\,\inf\lbrace\,
n\geq 0 : Z^{K+1}_n(0)=0
\,\rbrace\,.$$
Let $z\in\dD$ be such that $z_0\geq1$,
let the Markov chains $\Ztl$, $\Ztk$ start from $z$,
and let $z_0$ be the starting point of the Markov chains 
$(Y^\ell_n)_{n\geq 0}$, $(Y^1_n)_{n\geq 0}$.
The inequalities between the occupancy processes translate to the associated Markov chains as follows:
\begin{align*}
\forall n\in \lbrace\,0,\dots,\tau(Y^\ell)\,\rbrace\qquad
&Y^\ell_n\,\leq\, Z^\ell_n(0)\,\leq\,Z^{K+1}_n(0)\,\leq\,Y^1_n\,,\\
\forall n\in\lbrace\,0,\dots,\tau(Z^{K+1})\,\rbrace\qquad
& Z^{K+1}_n(0)\,\leq\,Y^1_n\,.
\end{align*}
The occupancy processes $\Ttl$, $\Ttk$,
along with the associated Markov chains
$(Y^\ell_n)_{n\geq 0}$, $(Y^1_n)_{n\geq 0}$,
have been studied in detail in~\cite{CerfWF}.
Thanks to the relations just stated,
we will be able to make use of many of the estimates derived in~\cite{CerfWF}.
Let $\t$ be $K+1$ or $\ell$
and let us call $\widetilde{V}$
the cost function associated to the Markov chain $(Y^\t_n)_{n\geq 0}$.
We will make use of the following results from~\cite{CerfWF}:

Let us define a function $\widetilde{F}:[0,1]\to[0,1]$ as follows:
$$\forall r\in[0,1]\qquad 
\widetilde{F}(r)=\exa\frac{\s r}{(\s-1)r+1}\,.$$ 
\begin{lemma}\label{costms}
Suppose that $\s\exa>1$.
For $s,t\in[0,1]$,
we have $\widetilde{V}(s,t)=0$ if and only if

$\bullet$ either $s=t=0$,

$\bullet$ or there exists $l\geq 1$ such that $t=\widetilde{F}^l(s)$,

$\bullet$ or $s\neq 0$, $t=\rho^*$.
\end{lemma}

Let $\tau(Y^\t)$ be the first time that the Markov chain $(Y^\t_n)_{n\geq 0}$
becomes null:
$$\tau(Y^\t)\,=\,\inf\lbrace\,
n\geq 0 : Y^\t_n=0
\,\rbrace\,.$$
\begin{proposition}\label{pers}
Let $a\in\,]0,+\infty[\,$ and let $i\in\um$.
The expected value of $\tau(Y^\t)$ starting from $i$
satisfies
$$\lim_{\genfrac{}{}{0pt}{1}{\ell,m\to\infty,\,q\to 0}{{\ell q} \to a}}\,
\frac{1}{m}\ln E(\tau(Y^\t)
\,|\, Y^\t_0=i)\,=\,\widetilde{V}(\rho^*_0,0)\,.$$
\end{proposition}

\subsection{Concentration near $\rho^*$}\label{Concen}
We show next that,
when $\s\exa>1$,
asymptotically,
the Markov chain $\Ztt$
concentrates in a neighbourhood of $\rho^*$.
Let us loosely describe the strategy we will follow.
The Markov chain $\Ztt$ is a perturbation
of the dynamical system associated to the map $F$.
The map $F$ has two fixed points:
0 and $\rho^*$.
The fixed point $0$ is unstable,
while $\rho^*$ is a stable fixed point.
The proof relies mainly on two different kind of estimates.
We estimate first the typical time the process $\Ztt$
needs to leave a neighbourhood of the region $\lbrace\,z\in\dD:z_0=0\,\rbrace$;
since the instability at $0$ concerns principally the dynamics of the master sequence,
we will be able to make use of the estimates developed in~$\cite{CerfWF}$
by means of the inequalities stated in section~\ref{Compms}.
We estimate then the time the process $\Ztt$ spends outside a neighbourhood of
the region $\lbrace\,z\in\dD:z_0=0\,\rbrace$ and $\rho^*$.
Since $\Ztt$ tends to follow the discrete trajectories given by
the dynamical system associated to $F$,
it cannot stay a long time outside such a neighbourhood.
This fact will be proved with the help of the large deviations principle stated in the previous section.
This estimate will help us to bound the number of excursions
outside a neighbourhood of $\rho^*$, as well as the length of these excursions.
We formalise these ideas in the rest of the section.
In order to simplify the notation,
from now on we omit the superscript $\t$ 
and we denote by $P_z$ and $E_z$ 
the probabilities and expectations for the Markov chain $(Z_n)_{n\geq 0}$ starting from $z\in\dD$.

Let us define
$$D_\d\,=\,
\big\lbrace\,
r\in\cD:
0<r_0<\d
\,\big\rbrace\,.$$

\begin{lemma}\label{logsing}
For all $\d>0$,
there exists $c>0$,
depending on $\d$,
such that, asymptotically,
for all
$z\in\dD$ such that $z_0\geq 1$,
we have
$$
P_{z}\big(Z_1(0)>0,\dots,
Z_{\lfloor c\ln m\rfloor-1}(0)>0,
Z_{\lfloor c\ln m\rfloor}\in m(\cD\setminus \overline{D_\d})
\big)
\,\geq\,
\frac{1}{m^{c\ln m}}\,.$$
\end{lemma}
\begin{proof}
Let $(Y_n^\ell)_{n\geq 0}$
be the Markov chain defined in section~\ref{Compms}.
Let $\tau(Y^\ell)$ be the first time that
the process $(Y_n^\ell)_{n\geq 0}$ becomes null:
$$\tau(Y^\ell)\,=\,\inf\big\lbrace\,
n\geq0:
Y_n=0
\,\big\rbrace\,.$$
By the remarks in section~\ref{Compms}
we can see that
\begin{multline*}
P_{z}\big(Z_1(0)>0,\dots,
Z_{\lfloor c\ln m\rfloor-1}(0)>0,
Z_{\lfloor c\ln m\rfloor}\in m(\cD\setminus \overline{D_\d})
\big)
\,\geq\\
P_{z}\big(Z_1(0)>0,\dots,
Z_{\lfloor c\ln m\rfloor-1}(0)>0,
Z_{\lfloor c\ln m\rfloor}\in m(\cD\setminus \overline{D_\d}),
\tau(Y^\ell)>\lfloor c\ln m\rfloor
\big)
\\\geq\,
P_{1}\big(Y^\ell_1>0,\dots,
Y^\ell_{\lfloor c\ln m\rfloor-1}>0,
Y^\ell_{\lfloor c\ln m\rfloor}>m(\rho^*_0-\d)
\big)\,.
\end{multline*}
As shown in lemma~7.7 of~\cite{CerfWF},
this last probability is bounded from below by
$1/m^{c\ln m}$,
which gives the desired result.
\end{proof}
We estimate next the length of a typical excursion of
$(Z_n)_{n\geq 0}$
outside a neighbourhood of $\lbrace\,z\in\dD:z_0=0\,\rbrace$ and $m\rho^*$.
For $\rho\in\cD$ and $\d>0$,
we define the $\d$-neighbourhood of $\rho$ by
$$U(\rho,\d)\,=\,\big\lbrace\,
r\in\cD:
|r-\rho|<\d
\,\big\rbrace\,.$$
\begin{lemma}\label{exc}
For all $\d>0$,
there exist $h\geq1$ and $c>0$,
depending on $\d$,
such that,
asymptotically,
for all $r\in\cD$ such that $r_0\geq \d$,
we have
$$P_{\lfloor mr\rfloor}
\big(
Z_1(0)>0,\dots,Z_{h-1}(0)>0,
Z_h\in mU(\rho^*,\d)
\big)
\,\geq\,
1-\exp(-cm)\,.$$
\end{lemma}
\begin{proof}
Let $\d>0$
and let us define the set 
$$\cK\,=\,\lbrace\,
r\in\cD\,:\, r_0\geq \d
\,\rbrace\,.$$
For each $r\in\cK$
there exists an integer $h_r\geq0$
such that $F^{h_r}(r)\in U(\rho^*,\d/4)$.
By continuity of the map $F$,
for each $r\in\cK$ there exist also positive numbers
$\d^r_0,\dots,\d^r_{h_r}$
such that 
$\d^r_0,\dots,\d^r_{h_r}<\d/2$
and
$$\forall k\in\lbrace\,0,\dots,h_{r}\,\rbrace\ \ 
F\big(
U(F^{k-1}(r),\d^r_{k-1})
\big)\,\subset\,
U(F^k(r),\d^r_{k}/2)\,.$$
The family $\lbrace\, U(r,\d^r_0)\,:\,r\in\cK \,\rbrace$
is an open cover of the set $\cK$;
since $\cK$ is a compact set,
we can extract a finite subcover, i.e.,
there exist $N\in\N$ and $r_1,\dots,r_N\in\cK$
such that 
$$\cK\,\subset\, U_0\,=\,\bigcup_{n=1}^N U(r_n,\d^{r_n}_0)\,.$$
Let us set $h=\max\lbrace\, h_{r_i}\,:\, 1\leq i\leq N\,\rbrace$,
For $n\in\lbrace\,1,\dots,N\,\rbrace$
we take $\d^{r_n}_{h_{r_n}+1},\dots,\d^{r_n}_{h}$
to be positive numbers such that, as before,
$$\forall k\in\lbrace\,h_{r_n}+1,\dots,h\,\rbrace\qquad
F\big(U(F^{k-1}(r_n),\d^{r_n}_{k-1})\big)\,\subset\, U(F^k(r_n),\d^{r_n}_k/2)\,.$$
Let us define
$$\forall k\in\lbrace\,1,\dots,h-1\,\rbrace\qquad
U_k\,=\,\bigcup_{n=1}^N U(F^k(r_n),\d^{r_n}_k)\,.$$
We have then, for any $r\in\cK$,
\begin{multline*}
P_{\lfloor mr\rfloor}\big(
Z_1(0)>0,\dots,Z_{h-1}(0)>0,
Z_h\in mU(\rho^*,\d)
\big)
\,\geq\\
P_{\lfloor mr\rfloor}\big(
\forall k\in \lbrace\,1,\dots,h\,\rbrace\quad
Z_k\in mU_k
\big)\,.
\end{multline*}
Passing to the complementary event,
\begin{multline*}
P_{\lfloor mr\rfloor}\big(
\exists k\in \lbrace\,1,\dots,h-1\,\rbrace\quad Z_k(0)=0
\ \text{ or }\ Z_h\not\in mU(\rho^*,\d)
\big)
\\\leq\,
P_{\lfloor mr\rfloor}\big(
\exists k\in \lbrace\,1,\dots,h\,\rbrace\quad
Z_k\not\in mU_k
\big)
\\\leq\,
\sum_{1\leq k\leq h}
P_{\lfloor mr\rfloor}\big(
Z_1\in mU_1,\dots,
Z_{k-1}\in mU_{k-1},
Z_k\not\in mU_k
\big)
\\\leq\,
\sum_{1\leq k\leq h}\
\sum_{z\in mU_{k-1}}
P_{\lfloor mr\rfloor}\big(
Z_{k-1}=z,
Z_k\not\in mU_k
\big)
\\\leq\,
\sum_{1\leq k\leq h}\
\sum_{z\in mU_{k-1}}
P_z\big(
Z_k\not\in mU_k
\big)
P_{\lfloor mr\rfloor}\big(
Z_{k-1}=z
\big)
\\\leq\,
\sum_{1\leq k\leq h}
\max\big\lbrace\,
P_z(Z_1\not\in mU_k):
z\in mU_{k-1}
\,\big\rbrace
\,.
\end{multline*}
We use now
the large deviations upper bound stated in proposition~\ref{pgdtrans1}.
We have, for all
$k\in\lbrace\,1,\dots,h\,\rbrace$,
\begin{multline*}
\limsup_{\genfrac{}{}{0pt}{1}{\ell,m\to\infty,\,q\to 0}{{\ell q} \to a}}
\frac{1}{m}\ln\max_{z\in mU_{k-1}}P_z\big(
Z_1\not\in mU_k
\big)
\,\leq\\
-\inf\big\lbrace\,
I(r,\xi,\b,t):
r\in \overline{U_{k-1}},\,
\xi\in\cD,\,
\b\in\cB(t),\,
t\not\in U_k
\,\big\rbrace\,.
\end{multline*}
For all 
$r\in \overline{U_{k-1}}$,
we have
$F(r)\in U_k$,
the previous infimum is thus strictly positive.
Since $h$ is fixed,
we conclude that
$$\limsup_{\genfrac{}{}{0pt}{1}{\ell,m\to\infty,\,q\to 0}{{\ell q} \to a}}
\frac{1}{m}\ln P_{\lfloor mr\rfloor}\big(
\exists k\in \lbrace 1,\dots,h-1\rbrace\ \ Z_k=0
\ \text{ or }\ 
Z_h\not\in mU(\rho^*,\d)
\big)
<0\,,$$
which finishes the proof of the lemma.
\end{proof}
\begin{corollary}\label{lexc}
Let $\d>0$.
There exist $h\geq 1$,
$c\geq 0$,
depending on $\d$,
such that,
asymptotically,
for all
$r\in\cD\setminus(\overline{D}_\d\cup U(\rho^*,\d))$
and for all $n\geq 0$,
we have
$$
P_{\lfloor mr\rfloor}\big(
Z_t\in\cD\setminus(\overline{D}_\d\cup U(\rho^*,\d))
\ \text{ for }\
0\leq t\leq n
\big)
\,\leq\,
\exp\Big(
-cm\Big\lfloor\frac{n}{h}\Big\rfloor
\Big)\,.
$$
\end{corollary}
The proof is carried out by dividing the interval $\lbrace\, 0,\dots,n\,\rbrace$
in subintervals of length $h$ 
and using the estimate of lemma~\ref{exc} on each of the subintervals.
We will not write the details, 
which can be found in the proof of corollary~7.10 of~\cite{CerfWF}.

\begin{proposition}\label{convchm}
Let 
$g:[0,1]\to[0,1]$
be an increasing and continuous function, 
such that $g(0)=0$.
For all $z^0\in\dD$ such that $z^0_0\geq 1$,
we have
$$\lim_{\genfrac{}{}{0pt}{1}{\ell,m\to\infty,\,q\to 0}{{\ell q} \to a}}
\frac{\displaystyle E\bigg(
\sum_{n=0}^{\tau_0-1}
g\bigg(
\frac{|Z_n|_1}{m}
\,\bigg|\,
Z_0=z^0
\bigg)
\bigg)}{E(\tau_0
\,|\,
Z_0=z^0)}
\,=\,
g(|\rho^*|_1)\,.$$
\end{proposition}
\begin{proof}
Let 
$\e>0$.
The function 
$g$
being continuous,
there exists
$\d>0$
such that
$$\forall \rho\in U(\rho^*,2\d)\qquad
\big|
g(|\rho|_1)-g(|\rho^*|_1)
\big|\,<\,\e\,.$$
We define next a sequence of stopping times
in order to control the excursions of the Markov chain
$(Z_n)_{n\geq 0}$
outside
$U(\rho^*,\d)$.
We take
$T_0=0$
and
\begin{align*}
T^*_1&=\inf\bigg\lbrace
n\geq0:
\frac{Z_n}{m}\in U(\rho^*,\d)
\bigg\rbrace\quad
&&T_1=\inf\bigg\lbrace
n\geq T^*_1:
\frac{Z_n}{m}\not\in U(\rho^*,2\d)
\bigg\rbrace\\
&\ \, \vdots &&\ \ \quad\vdots\\
T^*_k&=\inf\bigg\lbrace
n\geq T_{k-1}:
\frac{Z_n}{m}\in U(\rho^*,\d)
\bigg\rbrace\ 
&&T_k=\inf\bigg\lbrace
n\geq T^*_k:
\frac{Z_n}{m}\not\in U(\rho^*,2\d)
\bigg\rbrace\\
&\ \, \vdots &&\ \ \quad\vdots
\end{align*}
We have then
\begin{multline*}
\sum_{n=0}^{\tau_0-1}
g\bigg(
\frac{|Z_n|_1}{m}
\bigg)
-g(|\rho^*|_1)\tau_0
\,=\,
\sum_{k\geq 1}\
\sum_{n=T_{k-1}\wedge\tau_0}^{T^*_k\wedge\tau_0-1}
\bigg(
g\bigg(
\frac{|Z_n|_1}{m}
\bigg)
-g(|\rho^*|_1)
\bigg)
\\+\,
\sum_{k\geq 1}\
\sum_{n=T^*_k\wedge\tau_0}^{T_k\wedge\tau_0-1}
\bigg(
g\bigg(
\frac{|Z_n|_1}{m}
\bigg)
-g(|\rho^*|_1)
\bigg)\,.\hfil
\end{multline*}
Taking the absolute value, we obtain
$$
\bigg|
\sum_{n=0}^{\tau_0-1}
g\Big(
\frac{|Z_n|_1}{m}
\Big)
-g(|\rho^*|_1)\tau_0
\bigg|\,\leq\\
2g(1)\sum_{k\geq1}(T^*_k\wedge\tau_0-T_{k-1}\wedge\tau_0)
+\e\tau_0\,.
$$
We need to control the sum on the right hand side.
Let us define, for
$n\geq 0$,
$$N(n)\,=\,\max\big\lbrace\,
k\geq1:
T_{k-1}<n
\,\big\rbrace\,.$$
We can now rewrite the sum as follows:
$$\sum_{k\geq1}(T^*_k\wedge\tau_0-T_{k-1}\wedge\tau_0)\,=\,
\sum_{k=1}^{N(\tau_0)}(T^*_k\wedge\tau_0-T_{k-1})\,.$$
Let 
$\eta>0$
and let us take 
$t^\eta_m$
as in proposition~7.11 of~\cite{CerfWF}:
$$t^\eta_m\,=\,
\exp\big(m(\widetilde{V}(\rho^*_0,0)+\eta)\big)\,,$$
where $\widetilde{V}$ 
is the cost function governing the dynamics of the master sequence,
as defined in section~\ref{Compms}.
We decompose the previous sum as follows:
$$
\sum_{k=1}^{N(\tau_0)}(T^*_k\wedge\tau_0-T_{k-1})\,\leq\\
1_{\tau_0>t^\eta_m}\tau_0
+1_{\tau_0\leq t^\eta_m}\sum_{k=1}^{N(\tau_0)}(T^*_k\wedge\tau_0-T_{k-1})\,.
$$
Let 
$z^0\in\dD$
such that $z^0_0\geq 1$.
Since the estimates are the same for every starting point,
we do not write the starting point in what follows:
the probabilities and expectation are all taken with respect to the initial condition
${Z_0=z^0}$,
unless otherwise stated.
We take the expectation in the previous inequalities and we obtain
\begin{multline*}
\bigg|
E\bigg(
\sum_{n=0}^{\tau_0-1}
g\Big(
\frac{|Z_n|_1}{m}
\Big)
\bigg)
-g(|\rho^*|_1)E(\tau_0)
\bigg|
\,\leq\\
2g(1)E(1_{\tau_0>t^\eta_m}\tau_0)
+2g(1)
E\bigg(
1_{\tau_0\leq t^\eta_m}
\sum_{k=1}^{N(\tau_0)}
(T^*_k\wedge\tau_0-T_{k-1})
\bigg)+\e E(\tau_0)
\,.
\end{multline*}
Thanks to the estimates developed in section~7.3 of~\cite{CerfWF},
we know that
$$\lim_{m\to\infty}
E(1_{\tau_0>t^\eta_m}\tau_0)
\,=\,0\,.$$
Proceeding as in lemma~7.13 of~\cite{CerfWF},
we can obtain the following bound on $N$:
\begin{lemma}\label{nexc}
There exists 
$c>0$,
depending on
$\d$,
such that, asymptotically,
$$
\forall k,n\geq0\qquad
P(N(n)>k)\,\leq\,
\frac{n^k}{k!}\exp(-cmk)\,.
$$
\end{lemma}
We estimate next the term
$$E\Bigg(
1_{\tau_0\leq t^\eta_m}
\sum_{k=1}^{N(\tau_0)}
(T^*_k\wedge\tau_0-T_{k-1})
\Bigg)\,.$$
We have, by the Cauchy--Schwarz inequality,
\begin{multline*}
E\Bigg(
1_{\tau_0\leq t^\eta_m}
\sum_{k=1}^{N(\tau_0)}
(T^*_k\wedge\tau_0-T_{k-1})
\Bigg)
\,=\,
\sum_{k\geq 1}
E\big(
1_{\tau_0\leq t^\eta_m}
1_{k\leq N(\tau_0)}
(T^*_k\wedge\tau_0-T_{k-1})
\big)
\\\leq\,
\sum_{k\geq 1}
P\big(
\tau_0\leq t^\eta_m,\,
N(\tau_0)\geq k
\big)^{1/2}
E\big(
1_{k\leq N(\tau_0)}
(T^*_k\wedge\tau_0-T_{k-1})^2
\big)^{1/2}
\\\leq\,
\sum_{k\geq 1}
P\big(
N(t^\eta_m)\geq k
\big)^{1/2}
E\big(
1_{k\leq N(\tau_0)}
(T^*_k\wedge\tau_0-T_{k-1})^2
\big)^{1/2}\,.
\end{multline*}
If 
$1\leq k\leq N(\tau_0)$,
then
$T_{k-1}<\tau_0$
and
$Z_{T_{k-1}}(0)>0$.
Thanks to the Markov property,
\begin{multline*}
E\big(
1_{k\leq N(\tau_0)}
(T^*_k\wedge\tau_0-T_{k-1})^2
\big)
\\=\,
\sum_{z\in \dD:\, z_0\geq 1}
E\big(
1_{k\leq N(\tau_0)}
(T^*_k\wedge\tau_0-T_{k-1})^2
\,\big|\,
Z_{T_{k-1}}=z
\big)
\times
P\big(
Z_{T_{k-1}}=z
\big)
\\\leq\,
\sum_{z\in \dD:\, z_0\geq 1}
E_z\big(
(T^*_1\wedge\tau_0)^2
\big)
P\big(
Z_{T_{k-1}}=z
\big)\,.
\end{multline*}
We next seek an upper bound on the random time
$T^*_1\wedge\tau_0$,
whenever the Markov chain starts form 
$z\in\dD$
with $z_0\geq 1$.

\begin{lemma}\label{longexc}
For all $\d>0$,
there exist $h\geq 1$, $c>0$,
depending on $\d$,
such that,
asymptotically,
for $z\in\dD$ such that $z_0\geq 1$,
$$P_z\big(
Z_{\lfloor c\ln m\rfloor+h}\in mU(\rho^*,\d)
\big)
\,\geq\,
\frac{1}{2m^{c\ln m}}\,.$$
\end{lemma}
\begin{proof}
Thanks to lemma~\ref{logsing},
there exists $c>0$,
such that,
asymptotically,
for all
$z\in\dD$
such that $z_0\geq 1$,
$$P_z\big(
Z_{\lfloor c\ln m\rfloor}(0)>\d m 
\big)
\,\geq\,
\frac{1}{m^{c\ln m}}\,.$$
Likewise,
thanks to lemma~\ref{exc},
there exist
$h\geq 1$ and $c'>0$,
such that,
asymptotically,
for all $z'\in\dD$
such that $z'_0\geq\d m$,
$$P_{z'}\big(
Z_h\in mU(\rho^*,\d)
\big)
\,\geq\,
1-\exp(-c'm)\,.$$
Thus,
\begin{multline*}
P_z\big(
Z_{\lfloor c\ln m\rfloor+h}\in mU(\rho^*,\d)
\big)
\\\geq\,
P_z\big(
Z_{\lfloor c\ln m\rfloor}(0)\geq \d m,\,
Z_{\lfloor c\ln m\rfloor+h}\in mU(\rho^*,\d)
\big)
\,=\\
\sum_{z':z'_0\geq \d m}
P_z\big(
Z_{\lfloor c\ln m \rfloor}=z'
\big)
P_{z'}\big(
Z_h\in mU(\rho^*,\d)
\big)
\,\geq\,
\frac{1}{m^{c\ln m}}
\big(
1-\exp(-c'm)
\big)\,,
\end{multline*}
and the result holds.
\end{proof}
\begin{corollary}\label{loco}
For all
$\d>0$,
there exist
$h\geq 1$, $c>0$,
depending on $\d$,
such that,
asymptotically,
for all
$z\in\dD$
such that $z_0\geq 1$,
$$\forall n\geq 0\qquad
P_z\big(
T^*_1\wedge\tau_0\geq
n(\lfloor c\ln m\rfloor+h)
\big)
\,\leq\,
\bigg(
1-\frac{1}{2m^{c\ln m}}
\bigg)^n\,.$$
\end{corollary}
Again, the proof is done by dividing the interval
$\lbrace\,0,\dots,n(\lfloor c\ln m\rfloor+h)\,\rbrace$
in subintervals of length $\lfloor c\ln m\rfloor+h$
and using the estimates of lemma~\ref{longexc} on each subinterval,
as in corollary~7.10 of~\cite{CerfWF}.
Thanks to corollary~\ref{loco},
asymptotically,
for all $z\in\dD_K$
such that $z_0\geq 1$,
$$E_z\big(
(T^*_1\wedge\tau_0)^2
\big)
\,=\,
\sum_{k\geq 1}
P_z\big(
T^*_1\wedge\tau_0\geq \sqrt{k}
\big)
\,\leq\,
\sum_{k\geq 1}\bigg(
1-\frac{1}{2m^{c\ln m}}
\bigg)^{\Big\lfloor \textstyle\frac{\sqrt{k}}{\lfloor c\ln m\rfloor+h} \Big\rfloor}\,.$$
Let us set 
$$\a=1-\frac{1}{2m^{c\ln m}}\,,\qquad
t=\lfloor c\ln m\rfloor+h\,.$$
We have:
$$
\sum_{k\geq 1}\a^{\lfloor \sqrt{k}/t\rfloor}
\,\leq\,
\sum_{k\geq 1}\a^{\sqrt{k}/t-1}
\,\leq\,
\int_0^\infty \a^{\sqrt{x}/t-1}\,dx
\,=\,
\frac{2t^2}{\a(\ln\a)^2}\,.
$$
Therefore,
asymptotically,
for all $z\in\dD_K$
such that $z_0\geq 1$,
$$E_z\big(
(T^*_1\wedge\tau_0)^2
\big)
\,\leq\,
m^{3c\ln m}\,.$$
Thus, for all $k\geq 1$,
$$E\big(
1_{k\leq N(\tau_0)}
(T^*_k\wedge\tau_0-T_{k-1})^2
\big)
\,\leq\,
m^{3c\ln m}\,.$$
Together with lemma~\ref{nexc}, this implies that
\begin{multline*}
E\Bigg(
1_{\tau_0\leq t^\eta_m}
\sum_{k=0}^{N(\tau_0)}
\big(T^*_k\wedge\tau_0-T_{k-1}\big)
\Bigg)
\,\leq\,
\sum_{k\geq 1}
P\big(
N(t^\eta_m)>k
\big)^{1/2}
\big( m^{3c\ln m} \big)^{1/2}
\\\leq\,
m^{3c\ln m}
\Bigg(
t^\eta_m\exp(-cm/3)
+\sum_{k\geq t^\eta_m\exp(-cm/3)}
\bigg(
\frac{(t^\eta_m)^k}{k!}\exp(-cmk)
\bigg)^{1/2}
\Bigg)
\\\leq\,
m^{3c\ln m}
\Bigg(
t^\eta_m\exp(-cm/3)
+\sum_{k\geq0}\exp\Big(
\frac{k}{2}-cm\frac{k}{3}
\Big)
\Bigg)\,.
\end{multline*}
The last inequality holds since
$k!\geq (k/e)^k$.
We choose $\eta$
such that $0<\eta<c/3$.
Thanks to the preceding inequality,
\begin{multline*}
\limsup_{\genfrac{}{}{0pt}{1pt}{\ell,m\to\infty,\,q\to0}{\ell q\to a}}
\frac{1}{m}\ln E\Bigg(
1_{\tau_0\leq t^\eta_m}
\sum_{k=1}^{N(\tau_0)}\big(
T^*_k\wedge\tau_0-T_{k-1}
\big)
\Bigg)
\\\leq\,
\max\Big(
\widetilde{V}(\rho^*_0,0)+\eta-\frac{c}{3}
\Big)
\,<\,
\widetilde{V}(\rho^*_0,0)\,.
\end{multline*}
Theses estimates,
along with the result of proposition~\ref{pers},
imply that
$$\Bigg|
E\bigg(
\sum_{n=0}^{\tau_0-1}
g\bigg(
\frac{|Z_n|_1}{m}
\bigg)
\bigg)
-g(|\rho^*|_1)E(\tau_0)
\Bigg|\,\leq\,
3\e E(\tau_0)\,,$$
which concludes the proof of proposition~\ref{convchm}.
\end{proof}

\section{Synthesis}
The first statement of theorem~\ref{main}
is proved in~\cite{CerfWF}
for the case of the master sequence, $K=0$.
The proof for the case $K\geq 1$
does not involve any new arguments or ideas 
for a better understanding of the model;
it is a straightforward generalisation of the proof
for the case $K=0$.
Thus we deal only with the second statement of theorem~\ref{main}.
Let us suppose that 
$\a\psi(a)>\ln\k$.
As shown in~\cite{CerfWF},
the following estimates hold:
$\forall a\in\,]0,+\infty[\,$,
$\forall \a\in[0,+\infty]$,
\begin{align*}
\lim_{\genfrac{}{}{0pt}{1}{\ell,m\to\infty,\,q\to 0}{{\ell q} \to a,\,\frac{m}{\ell}\to\a}}
\frac{1}{m}\ln E(\tau_0\,|\,Z^\t_0=z^\t_{\text{enter}})\,&=\,\widetilde{V}(\rho^*,0)\,,\\
\limsup_{\genfrac{}{}{0pt}{1}{\ell,m\to\infty,\,q\to 0}{{\ell q} \to a,\,\frac{m}{\ell}\to\a}}
\frac{1}{\ell}\ln E(\tau^*\,|\, O^\t_0=o^\t_{\text{exit}})\,&\leq\,\ln\k\,.
\end{align*}
Thus,
since we are studying the case $\a\psi(a)>\ln\k$,
$$\lim_{\genfrac{}{}{0pt}{1}{\ell,m\to\infty,\,q\to 0}{{\ell q} \to a,\,\frac{m}{\ell}\to\a}}
\frac{E(\tau_0\,|\,Z^\t_0=z^\t_{\text{enter}})}
{E(\tau^*\,|\, O^\t_0=o^\t_{\text{exit}})}\,=\,+\infty\,.$$
On one hand, $g$ being a bounded function, the above identity readily implies that
$$\lim_{\genfrac{}{}{0pt}{1}{\ell,m\to\infty,\,q\to 0}{{\ell q} \to a,\,\frac{m}{\ell}\to\a}}
\frac{\displaystyle 
E\Bigg(
\sum_{n=0}^{\tau^*-1}
g\bigg(
\frac{|\pi(O^\t_n)|_1}{m}
\bigg)
\,\Bigg|\,
O^\t_0=\otex
\Bigg)}{E(\tau^*\,|\, O^\t_0=\otex)
+E(\tau_0\,|\, Z^\t_0=z^\t_{\text{enter}})}\,=\,0\,.
$$
On the other hand, using proposition~\ref{convchm}, we see that
$$
\lim_{\genfrac{}{}{0pt}{1}{\ell,m\to\infty,\,q\to 0}{{\ell q} \to a,\,\frac{m}{\ell}\to\a}}
\frac{\displaystyle
E\Bigg(
\sum_{n=0}^{\tau_0-1}
g\bigg(
\frac{|Z^\t_n|_1}{m}
\bigg)
\,\Bigg|\,
Z^\t_0=z^\t_{\text{enter}}
\Bigg)}{E(\tau^*\,|\, O^\t_0=\otex)
+E(\tau_0\,|\, Z^\t_0=z^\t_{\text{enter}})}\,=\,g(\rho^*_0+\cdots+\rho^*_K)\,.
$$
Reporting back in the formula at the very end of section~\ref{Invmes},
we conclude the proof of theorem~\ref{main}.

\bibliographystyle{plain}
\bibliography{wfart}
\end{document}